\documentclass[11pt]{scrartcl}
\usepackage{amsfonts,amssymb,amsthm,amsmath,booktabs}
\usepackage{color,graphicx}
\usepackage[latin1]{inputenc}

\ifx\pdftexversion\undefined
\usepackage{nohyperref}
\else
\usepackage[colorlinks=true,linkcolor={blue},citecolor={blue},urlcolor={blue},
  pdfauthor={Thomas Apel, Johannes Pfefferer, Sergejs Rogovs, Max Winkler},
  pdfstartview={Fit}]{hyperref}
\fi

\newcommand{\abssec}[1]{\noindent\normalsize {\bfseries #1\quad }\ignorespaces}
\renewenvironment{abstract}{\abssec{Abstract}}{\par\vspace{.1in}}
\newenvironment{keywords}{\abssec{Key Words}}{\par\vspace{.1in}}
\newenvironment{AMSMOS}{\abssec{AMS subject classification}}{\par\vspace{.1in}}

\theoremstyle{plain}
\newtheorem{theorem}{Theorem}

\newtheorem{lemma}[theorem]{Lemma}

\theoremstyle{definition}
\newtheorem{assumption}[theorem]{Assumption}
\newtheorem{remark}[theorem]{Remark}

\numberwithin{equation}{section}

\DeclareMathOperator{\supp}{supp}
\DeclareMathOperator{\dist}{dist}

\DeclareMathOperator{\cl}{cl}

\begin{document}

\title{\LARGE $L^{\infty}$-error estimates for Neumann boundary value problems on graded meshes}

\author{Thomas Apel\thanks{\texttt{thomas.apel@unibw.de},
    Universit\"at der Bundeswehr M\"unchen, Institut f\"ur Mathematik
    und Computergest\"utzte Simulation, D-85577 Neubiberg, Germany} 
    \and Johannes Pfefferer\thanks{\texttt{pfefferer@ma.tum.de},
    Chair of Optimal Control, Technical University of Munich,
    Boltzmannstra\ss e~3, D-85748 Garching b. M\"unchen, Germany}
    \and Sergejs Rogovs\thanks{\texttt{sergejs.rogovs@unibw.de},
    Universit\"at der Bundeswehr M\"unchen, Institut f\"ur Mathematik
    und Computergest\"utzte Simulation, D-85577 Neubiberg, Germany} 
    \and Max Winkler \thanks{\texttt{max.winkler@mathematik.tu-chemnitz.de},
    Chemnitz University of Technology, Professorship Numerical Mathematics 
    (Partial Differential Equations),
    D-09112 Chemnitz, Germany} 
    }
\maketitle

\begin{abstract}
This paper deals with a priori pointwise error estimates for the finite element solution of boundary value
problems with Neumann boundary conditions in
polygonal domains. Due to the corners of the domain, the convergence rate of
the numerical solutions can be lower than in case of smooth domains. As a
remedy the use of local mesh refinement 
near the corners is investigated. In order to prove quasi-optimal a priori error estimates regularity results in weighted Sobolev
spaces are exploited. 
This is the first work on the Neumann boundary value problem where both the regularity of the data is exactly specified and the sharp convergence order $h^2 \lvert \ln h \rvert $ 
in the case of piecewise linear finite element approximations is obtained.  As
an extension we show the same rate for the approximate solution of a
semilinear boundary value problem. The proof relies in this case on the supercloseness
between the Ritz projection to the continuous solution and the finite element solution.
\end{abstract}


\begin{AMSMOS}
  35J91, 65N15, 65N30, 65N50
\end{AMSMOS}

\section{Introduction}
The problem we investigate in the present paper reads  
\begin{equation*}
-\Delta y + y = f\quad\mbox{in}\ \Omega,\qquad \partial_n y = g\quad\mbox{on}\ \Gamma,
\end{equation*}
where $\Omega$ is some plane polygonal domain with boundary $\Gamma$.
Our aim is to derive a quasi-optimal error estimate for the piecewise linear finite element 
approximation of $y$ in the maximum norm. 
As the boundary $\Gamma$ is polygonal, there occur singularities in the solution,
which result in a reduced regularity of the solution, more precisely, 
the regularity assumption $y\in W^{2,\infty}(\Omega)$ 
used in many contributions in general does not hold if the maximal interior angle of the domain is equal to or greater than $90^\circ$, even if the input data are 
regular. 
However, $W^{2,\infty}(\Omega)$-regularity is required to obtain the full
order of convergence in the $L^\infty(\Omega)$-norm on quasi-uniform meshes. 
In order to achieve this in arbitrary domains, we use locally refined meshes as the circumstances require.

Let us give an overview of some fundamental 
contributions on maximum norm estimates for elliptic problems, where convergence rates for piecewise linear finite element approximations are considered.
Most  of those papers deal with  approximations on quasi-uniform meshes with maximal element diameter $h$.
In \cite{Nitsche_spline}, Nitsche showed the convergence rate of $h$ for the Dirichlet problem in convex polygonal domains for a right-hand side in $L^2(\Omega)$.
Under the assumption that the solution belongs to $W^{2,\infty}(\Omega)$, Natterer \cite{Natterer} 
showed the convergence rate of $h^{2-\varepsilon}$ with arbitrary $\varepsilon>0$.
This result was improved by Nitsche \cite{Nitsche} who showed the approximation order $h^2 \lvert \ln h \rvert^{3/2}$.
The sharp convergence rate $h^2\lvert\ln h\rvert$ has been finally shown by
Frehse and Rannacher \cite{Frehse} and by 
Scott \cite{Scott} for a slightly different problem satisfying Neumann boundary conditions.
Closely related is a recent contribution of Kashiwabara and Kemmochi
\cite{KashiKem18}, who
consider the Neumann problem and show
the same rate for an approximation which is non-conforming as
the smooth computational domain is replaced by a sequence of polygonal domains.
In case of domains with polygonal boundary, where the regularity of the solution
might be reduced, Schatz and Wahlbin \cite{Schatz1} showed the convergence rate 
$h^{\min\{2,\pi/\omega\}-\varepsilon}$ for the Dirichlet problem, where $\omega$ is the largest opening 
angle in the corners. In a further paper \cite{Schatz2} they 
improved the convergence rate to $h^{2-\varepsilon}$ by refining the mesh
towards the corners, which have opening angles larger than
$90^\circ$.  
An additional improvement for locally refined meshes is shown by Sirch \cite{Sirch}, who obtained the 
rate $h^2 \lvert \ln h \rvert^{3/2}$. Moreover, in that reference precise regularity
assumptions on the data are established, which, for instance, are required to
derive pointwise error estimates for optimal control problems involving a boundary value problem as a constraint.
Later on, several articles, see e.\,g.\
\cite{SchatzWeak,SchatzWahlbinQuasi,LeyVex16},
dealt with stability estimates (up to the factor $\lvert\ln h\rvert$) for the Ritz
projection.
This directly implies a quasi-best-approximation property in the maximum
norm and is in particular of interest for parabolic problems.
In our context these results can also be used to derive error estimates.
However, up to now there are no results of this kind available in the
literature for locally refined meshes.

In the present paper we discuss the Neumann problem.
Under the assumption that the mesh is refined appropriately near the corners where the solution fails to be $W^{2,\infty}$-regular, 
we show the estimate
\begin{equation*}
 \lVert y-y_h\rVert_{L^\infty(\Omega)} \le c h^2 \lvert \ln h \rvert.
\end{equation*}
This estimate contains several novelties and improvements in comparison to the
results known from the literature:
\begin{enumerate}
  \item This is the first contribution dealing
    with maximum norm estimates for the Neumann
    problem using locally refined meshes. 
    The proof differs essentially from the Dirichlet case, since, for instance, 
    Poincar\'e inequalities are not applicable. Moreover, in the presence of
    Neumann conditions weighted Sobolev spaces with nonhomogeneous instead of homogeneous norms have
    to be used. As a consequence, several interpolation error estimates and a
    priori bounds for the solution are different.
  \item Even for less regular solutions (due to the corner singularities) but on locally refined meshes, 
    we show that the exponent of the logarithmic term is equal to one. This
    exponent is known to be sharp for piecewise linear elements
    \cite{Haverkamp}. 
    With slight modifications our result can be applied to the Dirichlet problem as well.
    Although the paper \cite{Sirch_paper} claims an error estimate for the Dirichlet boundary value problem with the rate $h^2\lvert \ln h \rvert $, there is a mistake in \cite[Lemma 2.13]{Sirch_paper} fixed in \cite{Sirch}, which led to the error rate $h^2 \lvert \ln h \rvert^{3/2}$. Using the techniques of the present paper, one can guarantee the reduced exponent of the logarithmic term for the Dirichlet problem as well, see \cite{Rogovs}.
  \item We can specify the required regularity of the input data on the
right-hand side of the estimate. The paper is written in the spirit that 
the constant $c$ depends linearly on some (weighted) H\"older norm of $f$ and $g$.
As already mentioned above, such a result is necessary in order to get maximum
norm estimates for related optimal control problems. This application will be
documented in a forthcoming paper.
\item As a further application we derive quasi-optimal pointwise error
  estimates for the finite element approximation of a semilinear partial differential
  equation. For this purpose, we pick up a fundamental idea from \cite{Pfefferer}. 
  The key observation therein is a supercloseness result between
  the discrete solution and the Ritz projection of the continuous solution.
  With this intermediate result and the 
  quasi-optimal convergence rates for linear problems in the maximum norm shown in
  the present paper, we can
  easily obtain the quasi-optimal convergence rate for semilinear problems as in the linear setting.
\end{enumerate}
For the proof of our main result we combine multiple techniques. Near corners 
where the singularities are mild, i.\,e., where the solution still belongs to
$W^{2,\infty}$, we apply the result of Scott \cite{Scott} to some localized auxiliary
problem.
Otherwise, we apply the ideas from Schatz and Wahlbin \cite{Schatz2}
and introduce dyadic decompositions around the singular corners which allows us
to exactly carve out both the singular behavior of the solution and the local
refinement of the finite element mesh.
With local finite element error estimates in the maximum norm, e.\,g.\ the one
from \cite{Wahlbin},  we can then decompose
the error into a local quasi-best-approximation term and a finite element
error in a weighted $L^2(\Omega)$-norm, where the weight is a regularized
distance function towards the corners. This term is discussed using a duality argument
as well as local energy norm estimates on the dyadic decomposition. The
pollution terms
arising in local finite element error estimates are treated by a kick-back
argument. For the best-approximation terms we use tailored interpolation error
estimates exploiting regularity results in weighted Sobolev spaces. The
required regularity results are taken from \cite{Kozlov2,Kozlov,Plamen,MR10,Nazarov}.

The paper is structured as follows. In Section 2 we introduce the notation and the function spaces that we use. 
Moreover, we recall a regularity result in weighted Sobolev spaces. 
We establish and prove the main result, namely the maximum norm estimate for the finite element approximation of the Neumann problem, 
in Section 3. The application of this result to semilinear problems is presented in Section 4.
In Section 5 we confirm by numerical experiments that the proven maximum norm estimate is sharp. 

We notice that, throughout the paper, $c>0$ is a generic constant independent of the mesh size, and may have a different value at each occurrence.

\section{Notation and regularity}\label{sec:regularity}
Throughout this paper $\Omega$ is a bounded, two dimensional domain with polygonal boundary $\Gamma$. The corner points of $\Omega$ are denoted by $x^{(j)},$ $j=1,\ldots m$, and are numbered counter-clockwise. 
Moreover, $\Gamma_j$ is the edge of the boundary $\Gamma$ which connects the corner points $x^{(j)}$ and $x^{(j+1)}$, and we define $x^{(m+1)} = x^{(1)}.$ The interior angle between $\Gamma_{j-1}$ and $\Gamma_j$ 
is denoted by $\omega_j$ with the obvious modification for $\omega_1$. Furthermore, we denote by $r_j$ and $\varphi_j$ the polar coordinates located at the point $x^{(j)}$ such that $\varphi_j=0$ on the edge $\Gamma_j$.

In this paper we derive a maximum norm error estimate for the finite element discretization of the Neumann problem
\begin{equation}\label{problem}
\begin{split}
 -\Delta y + y &= f\quad  \text{in}\ \Omega,\\
 \partial_n y &= g\quad   \text{on}\ \Gamma
 \end{split}
\end{equation}
with input data $f\in L^2(\Omega)$ and $g \in L^2(\Gamma)$. 
Later on, we will require higher regularity assumptions on the data in order to derive the quasi-optimal pointwise discretization error estimates. These are stated when needed.
The variational solution of \eqref{problem} is the unique element $y\in H^1(\Omega)$ which satisfies 
\begin{equation}\label{var_cont}
 a(y,v)=(f,v)_{L^2(\Omega)}+(g,v)_{L^2(\Gamma)}\quad \forall v\in V := H^1(\Omega),
\end{equation}
where $a:H^1(\Omega)\times H^1(\Omega) \rightarrow \mathbb{R}$ is the bilinear form defined by
\begin{equation}\label{eq:bilinear}
 a(y,v):=\int_{\Omega} (\nabla y \cdot \nabla v + y v).
\end{equation}	

It can be shown \cite{Grisvard} that the regularity of the solution $y$ of the boundary value problem \eqref{problem} near $x^{(j)}$ is characterized 
by the eigenvalues of an operator pencil generated by the Laplace operator in an infinite cone, which coincides with $\Omega$ near the corner $x^{(j)}.$ 
In our case, the leading eigenvalues are explicitly known to be $\lambda_j:=\pi/\omega_j$.
If $\lambda_j\notin\mathbb N$, the corresponding singular functions have the form
\[
c_jr_j^{\lambda_j}\cos(\lambda_j \varphi_j)
\]
with certain stress-intensity factors $c_j \in \mathbb{R}$. The singular functions are slightly different if $\lambda_j \in \mathbb{N}$. For a more intensive discussion on this we refer to \cite[Section 4.4]{Grisvard} and \cite[Section 2.\S4]{Nazarov}.

To capture these singular parts in the solution accurately, we use adapted function spaces
containing weight functions of the form $r_j^{\beta_j}$. 
To this end, we introduce for each $j=1,\ldots,m$ a circular sector $\Omega_{R_j}$,
\begin{equation*}
\Omega_{R_j}:=\{x \in \Omega \colon \lvert x - x^{(j)} \rvert < R_j \}
\end{equation*}
with radius $R_j>0$ centered at the corner $x^{(j)}$.
The radii $R_j$ can be chosen arbitrarily with the only restriction that the circular sectors $\Omega_{R_j}$ do not overlap for $j=1,\hdots, m$.
Furthermore, we require subsets depending on $i\in\mathbb{N}$ excluding the circular sectors $\Omega_{R_j/i}$ that we denote by
\[
 	\tilde\Omega_{R/i}:= \Omega\setminus\bigcup_{j=1}^m \Omega_{R_j/i}.
\] 
For $k\in \mathbb{N}_0$, $p \in [1,\infty]$ and $\vec{\beta} \in \mathbb{R}^m$
the weighted Sobolev spaces $W_{\vec{\beta}}^{k,p}(\Omega)$
are defined as the set of all functions in $\Omega$ with the finite norm
\begin{align*}
\lVert v \rVert_{W_{\vec{\beta}}^{k,p}(\Omega)} &= \lVert  v \rVert_{W^{k,p}(\tilde\Omega_{R/2})} + \sum_{j=1}^m \lVert  v \rVert_{W_{\beta_j}^{k,p}(\Omega_{R_j})}.
\end{align*}
Here, $W^{k,p}(\Omega)$ ($=H^k(\Omega)$ for $p=2$) are the classical Sobolev spaces. The weighted parts in the norms are defined by
\begin{align*}
 \lVert v \rVert_{W_{\beta_j}^{k,p}(\Omega_{R_j})}&:= \Bigg( \sum_{|\alpha| \leq k} \lVert r_j^{\beta_j} D^{\alpha} v \rVert^p_{L^p(\Omega_{R_j})} \Bigg)^{1/p}
 \end{align*}
for $1 \leq p<\infty$ and
\begin{align*}
 \lVert v \rVert_{W_{\beta_j}^{k,\infty}(\Omega_{R_j})}&:= \max_{|\alpha| \leq k} \lVert r_j^{\beta_j} D^{\alpha} v \rVert_{L^{\infty}(\Omega_{R_j})}
\end{align*}
for $p = \infty$.  The trace space of $W_{\vec{\beta}}^{k,p}(\Omega)$ for
$p\in [1,\infty)$ is denoted by $W^{k-1/p,p}_{\vec\beta}(\Gamma)$ and 
is equipped with the norm
\[
\lVert v \rVert_{W_{\vec{\beta}}^{k-1/p,p}(\Gamma)}:=\inf \Big\{ \lVert u \rVert_{W_{\vec{\beta}}^{k,p}(\Omega)} :  u \in W_{\vec{\beta}}^{k,p}(\Omega)\ \text{and}\ u\big|_{\Gamma\setminus \mathcal{C}} = v \Big\}
\]
with $\mathcal{C}:=\{x^{(1)},\ldots,x^{(m)}\}$, see \cite[Section 7]{Kozlov}.

Now, we recall a priori estimates in the weighted $H^2(\Omega)$-norm. 
Comparable results can be found in e.g. \cite{Solonnikov}, \cite{Plamen}, \cite[Section 4.5]{Nazarov}, \cite[Section 7]{Kozlov}. 
However, due to similarities of the considered problems as well as of the notation, we cite the result from \cite[Lemma 3.11]{Pfefferer}.
{\lemma\label{lemma_beta}
Let $\vec\beta\in [0,1)^m$ satisfy the condition
$1-\lambda_j < \beta_j$, $j=1,\ldots,m$.
For every $f \in W_{\vec{\beta}}^{0,2}(\Omega)$
and $g \in W_{\vec{\beta}}^{1/2,2}(\Gamma)$, the solution of problem \eqref{var_cont} belongs to $W_{\vec{\beta}}^{2,2}(\Omega)$ which satisfies the a priori estimate
\[
 \lVert y \rVert_{W_{\vec{\beta}}^{2,2}(\Omega)} 
 \leq c\left( \lVert f \rVert_{W^{0,2}_{\vec{\beta}}(\Omega)} + \lVert g \rVert_{W^{1/2,2}_{\vec{\beta}}(\Gamma)} \right).
\]
}

\begin{remark}\label{rem:W2inf_regularity}
  For the pointwise error analysis we have to guarantee
  $y\in W^{2,\infty}_{\vec\gamma}(\Omega)$ with certain weights $\vec\gamma\in
  [0,2)^m$. In order to show this, one typically uses regularity results in weighted H\"older spaces.
  One possibility is an application of the theory in 
  weighted $N$-spaces introduced for instance in \cite[Chapter 4, Section
  \S5.5]{Nazarov} and \cite[Theorem 1.4.5]{Kozlov2}. 
  Based on this, it is shown in \cite[Lemma 3.13]{Pfefferer} that 
  $y$ belongs to $W^{2,\infty}_{\vec\gamma}(\Omega)$ and fulfills
  \[
  \|y\|_{W^{2,\infty}_{\vec\gamma}(\Omega)} \le c \left(\|f\|_{N^{0,\sigma}_{\vec\delta}(\Omega)}
    + \|g\|_{N^{1,\sigma}_{\vec\delta}(\Gamma)}\right)
  \]
  provided that the assumption
  \begin{equation}\label{eq:assumption_gamma}
    \left\lbrace
    \begin{array}{rl}
      \vec\gamma\in [0,2)^m\quad\mbox{ with }\quad &\gamma_j > 2-\lambda_j,\\
      \vec\delta\in[\sigma,2+\sigma)^m\quad \mbox{ with }\quad &\delta_j=\gamma_j+\sigma,
    \end{array}\right.\qquad
    j= 1,\ldots,m,
  \end{equation}
  is fulfilled.

  A further possibility is to use regularity results in weighted $C$-spaces from e.g.
  \cite[Chapter 4, Section \S5.5]{Nazarov} or \cite[Section 8.3]{MR10}.
  These spaces are more suitable for the inhomogeneous Neumann problem as
  $N^{1,\sigma}_{\vec\delta}(\Gamma)$ does not 
  contain constant functions if $\delta_j < 1+\sigma$ for some $j=1,\ldots,m$, see \cite[Lemma 6.7.5]{MR10}.
  However, to the best of our knowledge, regularity results in 
  weighted $C$-spaces are not directly accesible for our setting in the
  literature, but can be deduced with similar arguments as in \cite[Lemma 3.13]{Pfefferer}.
  Related results in case of polyhedral domains ($n=3$) are already shown in 
  \cite[Theorem 8.3.1]{MR10}.
\end{remark}

\section{Finite element error estimates}\label{sec:fem}
In this section we prove the first main result of this paper, namely the $L^{\infty}(\Omega)$-norm error estimate for the finite element approximation 
of boundary value problem \eqref{problem}.
To this end, we introduce a family of graded triangulations $\{\mathcal{T}_h\}_{h>0}$ of $\Omega$.
The global mesh parameter is
denoted by $h< 1$. As we want to obtain a quasi-optimal error estimate for arbitrary polygonal domains, we consider 
locally refined meshes and denote by $\mu_j \in (0,1]$, $j=1,\ldots,m$, the mesh grading parameters which are collected in the vector $\vec{\mu}\in (0,1]^m$. 
The distance between a triangle $T \in \mathcal{T}_h$ and the corner $x^{(j)}$ is defined by
\[ r_{T,j} := \inf_{x \in T} \lvert x - x^{(j)} \rvert.\]
We assume that for $j=1,\dots,m$ the element size $h_T:=\text{diam}(T)$ satisfies
\begin{equation}\label{mesh_cond}
\begin{aligned}
c_1 h^{1/\mu_j} &\leq h_T \leq c_2 h^{1/\mu_j}&\quad& \text{if}\ r_{T,j} = 0,\\
c_1 hr_{T,j}^{1-\mu_j} &\leq h_T \leq c_2 hr_{T,j}^{1-\mu_j}&& \text{if}\ 0 < r_{T,j} < R_j,\\\notag 
c_1 h &\leq h_T \leq c_2 h&& \text{if}\  r_{T,j} > R_j,
\end{aligned}
\end{equation}
with some constants $c_1,c_2 > 0$ independent of $h$ and refinement radii $R_j>0$, $j=1,\ldots,m$.
Such meshes are known for instance from \cite{Ogan,Raugel,Schatz2}.
For the finite element discretization we use the space of continuous and piecewise linear
functions in $\overline\Omega$, this is
\begin{equation}\label{V_h}
 V_h:=\{ v_h \in C(\overline{\Omega}) : v_h|_{T} \in \mathcal{P}_1\ \text{for all}\ T \in \mathcal{T}_h \}.
\end{equation}
The finite element solution $y_h\in V_h$  satisfies
\begin{equation}\label{var_disc}
 a(y_h,v_h) = (f,v_h)_{L^2(\Omega)}+ (g,v_h)_{L^2(\Gamma)}\ \ \ \forall v_h \in V_h.
\end{equation}

Under the assumption that the solution belongs to $W^{2,\infty}(\Omega)$ the desired
convergence rate for the solution of \eqref{var_disc}
holds on quasi-uniform meshes, see Scott \cite{Scott}.
We apply this result in our proof locally, near those corners, where the solution still belongs to $W^{2,\infty}(\Omega)$. The global estimate reads as follows.
 \begin{theorem}\label{thscott}
 Assume that the solution $y$ of \eqref{var_cont} 
 belongs to $W^{2,\infty}(\Omega)$ and that $\Omega$ is convex.
 Let $y_h \in V_h$ be the solution of \eqref{var_disc}.
 Then, the finite element error can be estimated by 
 \begin{equation}
  \lVert y-y_h \rVert_{L^{\infty}(\Omega)} \leq c h^2 \lvert \ln h\rvert \lVert y \rVert_{W^{2,\infty}(\Omega)} 
 \end{equation}
on a quasi-uniform sequence of meshes ($\vec\mu=\vec 1$).
 \end{theorem}
 
The following error estimate in the $L^2(\Omega)$-norm on graded meshes for the Neumann boundary value problem is shown in \cite[Lemma 3.41]{Pfefferer}.
{\lemma\label{lemma_pfefferer}
Let $y$ and $y_h$ be the solutions of \eqref{var_cont} and \eqref{var_disc}, respectively. 
It is assumed that $f \in W^{0,2}_{\vec{\beta}}(\Omega)$ and $g \in W^{1/2,2}_{\vec{\beta}}(\Gamma)$
with a weight vector $\vec\beta\in [0,1)^m$.
Then, the estimate
\begin{align*}
&\lVert y-y_h \rVert_{L^2(\Omega)} \ \leq c h^2  \lVert y \rVert_{W^{2,2}_{\vec{\beta}}(\Omega)}
\leq c h^2 \Big( \lVert f \rVert_{W^{0,2}_{\vec{\beta}}(\Omega)} + \lVert g \rVert_{W^{1/2,2}_{\vec{\beta}}(\Gamma)} \Big),
\end{align*}
is fulfilled, provided that $1-\lambda_j < \beta_j \leq 1-\mu_j$, $j=1,\ldots,m$.
}

Now we state the main theorem of this paper.
\begin{theorem}\label{th}
 Assume that $y$, the solution of \eqref{var_cont},
 belongs to $W^{2,\infty}_{\vec\gamma}(\Omega)$ with $\vec\gamma\in [0,2)^m$. 
 Moreover, let one of the following conditions be fulfilled:
 \begin{equation}\label{eq:th_assumptions}
   \begin{aligned}
&(i)\quad 0\le 2-\lambda_j < \gamma_j < 2-2\mu_j,\\
&(ii)\quad \lambda_j > 2,\ \gamma_j=0\ \mbox{and}\ \mu_j=1,
\end{aligned}
\end{equation}
for $j=1,\ldots,m$. Then, the solutions $y_h$ of \eqref{var_disc}
 satisfy the error estimate
 \[
  \lVert y-y_h \rVert_{L^{\infty}(\Omega)}\leq c h^2 \lvert \ln h\rvert  
  \lVert y \rVert_{W_{\vec{\gamma}}^{2,\infty}(\Omega)}.
 \]
\end{theorem}
In Remark \ref{rem:W2inf_regularity} we have already discussed several assumptions 
on the input data which imply the regularity for $y$ required in Theorem \ref{th}.
In particular, the range of feasible weights $\vec\gamma$ is non-empty if 
$\mu_j < \lambda_j/2$ for all $j=1,\ldots,m$ with $\omega_j \ge \pi/2$, and
otherwise $\mu_j =1$.

The remainder of this section is devoted to the proof of Theorem \ref{th}.
We distinguish among three cases, depending on the point $x_0$ where $\lvert y-y_h\rvert$
attains its maximum.
If $x_0$ is located near a corner, namely in $\Omega_{R_j/16}$ for some
$j=1,\ldots,m$, we discuss the cases:
\begin{enumerate}
 \item The triple $(\lambda_j,\gamma_j,\mu_j)$ satisfies \eqref{eq:th_assumptions} (i). 
   In this case we prove the desired estimate using a technique of Schatz and Wahlbin \cite{Schatz2}, this is, 
   we introduce a dyadic decomposition of $\Omega_{R_j}$ around the singular corner, and apply local estimates on each subset, where the meshes 
   are locally quasi-uniform.
 \item  The triple $(\lambda_j,\gamma_j,\mu_j)$ satisfies
   \eqref{eq:th_assumptions} (ii). 
   Due to $y\in W^{2,\infty}(\Omega_{R_j})$ we
   can then apply the estimate from Theorem \ref{thscott} 
   for a localized problem near the corner. 
 \end{enumerate}
The remaining case is:
\begin{enumerate}
\item[3.] The maximum is attained in $\tilde\Omega_{R/16}$. Here, we use interior maximum norm estimates, e.\,g.\ from \cite[Theorem 10.1]{Wahlbin}, and 
exploit higher regularity in the interior of the domain.
\end{enumerate}

\textbf{Case 1:} $x_0\in\Omega_{R_j/16}$ with $(\lambda_j,\gamma_j,\mu_j)$ satisfying
\eqref{eq:th_assumptions} (i).
For the further analysis we assume that $x^{(j)}$ is located at the origin and $R_{j}=1$. Furthermore, we suppress the subscript $j$ and write $\Omega_R = \Omega_{R_{j}}$, $\mu=\mu_j$, etc.
Analogous to \cite{Schatz2} we introduce a dyadic decomposition of $\Omega_R$,
\[
 \Omega_J=\{ x\in \Omega: d_{J+1} \leq |x|\leq d_J \},\quad J=0,\ldots,I,
\]
with $d_J:=2^{-J}$ for $J=0,\ldots,I$ and $d_{I+1}=0$.
Obviously, there holds
\begin{equation}\label{dyadic}
 \Omega_R = \bigcup_{J=0}^{I} \Omega_J, 
\end{equation} 
see also Figure \ref{part_of_omega}.
The largest index $I$ is chosen such that $d_I = c_I h^{1/\mu}$
with a mesh-independent constant $c_I\geq1$. 
This constant is specified in the proof of Lemma \ref{lemma_l2} where a kick-back argument is applied,
which holds for sufficiently large $c_I$ only. We hide it in the generic constant if there is no need in it.
\begin{figure}
\centering
\includegraphics[width=.3\textwidth]{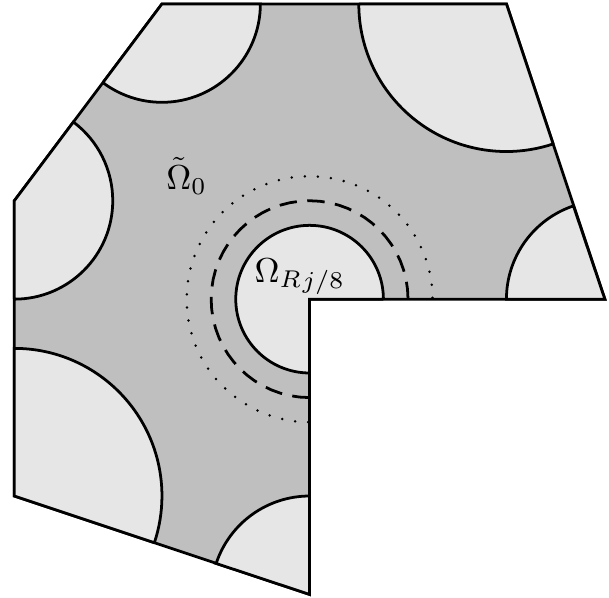}
\hspace{2cm}\includegraphics[width=.3\textwidth]{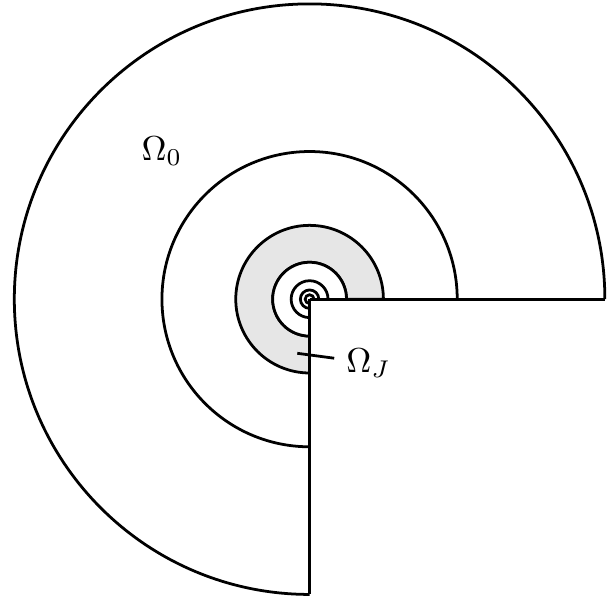}
  \caption{Partition of $\Omega$ in subdomains $\tilde{\Omega}^0$ and $\Omega_{R_j/8}$ (left) and partition of $\Omega_R$ in subdomains $\Omega_J$ (right)}\label{part_of_omega}
\end{figure}

We also  introduce the extended domains $\Omega_J'$ for $J \geq 1$ and $\Omega_J''$ for $J \geq 2$ by
\[
 \Omega_J'= \Omega_{J-1} \cup \Omega_J \cup \Omega_{J+1},\quad
 \Omega_J''= \Omega_{J-1}' \cup \Omega_J' \cup \Omega_{J+1}'
\]
with the obvious modifications for $J=I-1,I$. Obviously, the meshes $\mathcal T_h$ are locally quasi-uniform with the mesh sizes 
\[
 h_T \sim h_J:= h d_J^{1-\mu}\quad \mbox{if}\quad T\cap\Omega_J''\ne\emptyset
\]
for $J=2,\ldots,I$. This allows us to deduce local error estimates presented in the sequel of this paper.

For the convenience of the reader, we briefly summarize the forthcoming considerations. 
In Lemma \ref{lemma_inf} we show local $L^{\infty}$-norm error estimates on the subsets $\Omega_J$ where the underlying meshes are locally quasi-uniform. 
We distinguish between two cases.
In subdomains $\Omega_J$ for $J  > I-2$ we can use a local maximum norm estimate from \cite[Theorem 10.1]{Wahlbin}, 
and for $J = I-2,I-1,I$ we use a different approach based on an inverse inequality which we prove in Lemma \ref{lemma_v_est}. 
Both techniques allow a local decomposition of the finite element error into a best-approximation term, for which we apply interpolation error estimates that we recall in Lemma \ref{lemma_int},
and a pollution term. 
The pollution term arises as a weighted $L^2$-error which we discuss in Lemma \ref{lemma_l2}. For the proof of this estimate 
we also require local error estimates in $H^1(\Omega_J)$ stated in Lemma \ref{lemma_H1}.
 
{\lemma\label{lemma_v_est}
For $v_h \in V_h$ and $J = I-2,I-1,I$ there is the estimate
\[
 \lVert v_h \rVert_{L^{\infty}(\Omega_{J})} \leq c d_J^{-1} \lVert v_h \rVert_{L^2(\Omega'_{J})}.
\]}
\begin{proof} 
We denote by $T_*$ the element where $\lvert v_h \rvert$ attains its maximum
within $\Omega_J$ and by $F_{T^*}\colon \hat{T}\to T^*$ the affine transformation 
from the reference element $\hat T$ to $T^*$. Moreover, we use the notation
$\hat v_h(\hat x) := v_h(F_{T^*}(\hat x))$ for $\hat x\in\hat T$.
By this transformation and norm equivalences in finite-dimensional spaces, we have
\begin{align*}
\lVert v_h \rVert_{L^{\infty}(\Omega_J)} &\leq \lVert v_h \rVert_{L^{\infty}(T_*)} =
\lVert \hat{v}_h \rVert_{L^{\infty}(\hat{T})} 
 \leq c \lVert \hat{v}_h \rVert_{L^2(\hat{T})}\\
 &\leq c   h_{T_*}^{-1} \lVert v_h \rVert_{L^2(\Omega_J')} 
 \leq c d_J^{-1} \lVert v_h \rVert_{L^2(\Omega_J')} ,
\end{align*}
which proves the desired result, since $h_{T_*} \geq c h^{1/\mu}\sim d_I \sim d_J$ for $J=I-2,I-1,I$. 
\end{proof}

Next, we consider some error estimates for the nodal interpolant $I_h\colon C(\overline\Omega)\to V_h$.
The following results on graded meshes are taken from \cite[Lemma 3.58]{Pfefferer}, see also \cite[Lemma 3.7]{APR}.
{\lemma \label{lemma_int}
Let $p \in [2,\infty]$ and $l\in\{0,1\}$.
\begin{itemize}
\item[(i)] For $J=1,\ldots,I-2$ the estimates 
\begin{align}\label{L2IntO}
\lVert v - I_h v \rVert_{W^{l,2}(\Omega_J)} &\leq c h^{2-l} d_J^{(2-l)(1-\mu)+1-2/p-\beta}\lvert v \rvert_{W^{2,p}_{\beta}(\Omega'_J)},\\\label{LinftyIntO}
\lVert v - I_h v \rVert_{L^{\infty}(\Omega_J)} &\leq c h^{2-2/p} d_J^{(2-2/p)(1-\mu)-\beta}\lvert v \rvert_{W^{2,p}_{\beta}(\Omega'_J)}
\end{align}
are valid if $v \in W^{2,p}_{\beta}(\Omega'_J)$ with $\beta \in \mathbb{R}.$
\item[(ii)] Let $\theta_l := \max \{ 0, (3-l-2/p)(1-\mu) - \beta \}$ and $\theta_{\infty} := \max \{ 0, (2-2/p)(1-\mu) - \beta \}$. 
For $J=I,I-1$ the inequalities 
\begin{align}\label{L2IntI}
\lVert v - I_h v \rVert_{W^{l,2}(\Omega_J)} &\leq c c_I^{\theta_l +1 -2/p} h^{(3-l-2/p-\beta)/\mu} \lvert v \rvert_{W^{2,p}_{\beta}(\Omega'_J)},\\\label{LinftyIntI}
\lVert v - I_h v \rVert_{L^{\infty}(\Omega_J)} &\leq c c_I^{\theta_{\infty}} h^{(2-2/p-\beta)/\mu}  \lvert v \rvert_{W^{2,p}_{\beta}(\Omega'_J)}
\end{align}
hold if $v \in W^{2,p}_{\beta}(\Omega'_J)$ with $2/p-2 < \beta < 2-2/p.$
\end{itemize}}
\begin{remark}
  Lemma \ref{lemma_int} remains valid when replacing $\Omega_J$ by $\Omega_J'$ and $\Omega_J'$ by $\Omega_J''$, respectively. In this case the index range in part $(i)$ is $J=2,\ldots,I-3$,
  and in part $(ii)$ $J=I-2,\ldots,I$.
\end{remark}

The next result is needed in the proofs of Lemma \ref{lemma_inf} and Lemma \ref{lemma_l2}. It follows directly from \cite[Lemma 3.60]{Pfefferer}, see also \cite[Lemma 3.9]{APR}.
{\lemma\label{lemma_H1}
The following assertions hold:
\begin{itemize}
\item[(i)] For $J=2,\ldots,I-3$ the estimate 
\[
\lVert y - y_h \rVert _{H^1(\Omega_J)} \leq c \Big( h d_J^{2 \varepsilon + \mu}  \lvert y \rvert_{W^{2,\infty}_{\gamma}(\Omega''_J)} + d_J^{-1} \lVert y - y_h \rVert _{L^2(\Omega_J')} \Big) 
\]
is valid if $y\in W^{2,\infty}_{\gamma}(\Omega''_J)$ with $0 \leq \gamma \leq 2-2\mu -2 \varepsilon$ and sufficiently small $\varepsilon\ge 0$.
\item[(ii)] For $J=I-2,\ldots,I$ the inequality
\[
\lVert y - y_h \rVert _{H^1(\Omega_J)}  \leq c \Big( c_I^5  h^2  \lvert y \rvert_{W^{2,\infty}_{\gamma}(\Omega''_J)} + d_J^{-1} \lVert y - y_h \rVert _{L^2(\Omega_J')} \Big) 
\]
holds true if $y\in W^{2,\infty}_{\gamma}(\Omega''_J)$ with $0 \leq \gamma \leq 2-2\mu $.
\end{itemize}
}

In the next lemma we show local error estimates in the $L^{\infty}$-norm. 

{
\lemma\label{lemma_inf} For $y \in W_{\gamma}^{2,\infty}(\Omega''_J)$ with $0 \leq  \gamma \leq 2-2\mu$ the estimates
\begin{align}\label{Linfty1}
\lVert y - y_h  \rVert_{L^{\infty}(\Omega_J)} &\leq c \Big( h^2 \lvert \ln h\rvert \lvert y \lvert_{W^{2,\infty}_{\gamma}(\Omega''_J)}+ d_J^{-1} \lVert y - y_h  \rVert_{L^{2}(\Omega'_J)} \Big)\ &&for \ 2 \leq J< I-2,\\\notag
\lVert y - y_h  \rVert_{L^{\infty}(\Omega_J)} &\leq c \Big(  h^2 \lvert y \rvert_{W^{2,\infty}_{\gamma}(\Omega''_J)} + d_J^{-1}\lVert  y -  y_h \rVert_{L^2(\Omega'_J)} \Big)\ &&for \ J \geq I-2
\end{align}
are valid.  
}
\begin{proof}
Let us first consider the case $J<I-2$. From Theorem 10.1 and Example 10.1 in \cite{Wahlbin} the estimate
\begin{equation}\label{Schatz}
\lVert y - y_h  \rVert_{L^{\infty}(\Omega_J)} \leq c \Big( \lvert \ln h\rvert \inf_{\chi \in V_h} \lVert y - \chi \rVert_{L^{\infty}(\Omega'_J)} + d_J^{-1} \lVert y - y_h  \rVert_{L^{2}(\Omega'_J)} \Big)
\end{equation}
can be derived.
Estimate \eqref{Linfty1} in case of $2\le J < I-2$ follows from \eqref{Schatz} and \eqref{LinftyIntO} with $p=\infty$ exploiting $\gamma \leq 2-2\mu$, which provides
\begin{equation*}
\lVert y - I_h  y \rVert_{L^{\infty}(\Omega'_J)} \leq c h^2  d_J^{2-2 \mu - \gamma} \lvert y \rvert_{W^{2,\infty}_{\gamma}(\Omega''_J)} \leq c h^2 \lvert y \rvert_{W^{2,\infty}_{\gamma}(\Omega''_J)}.
\end{equation*}
For the case  $J=I,I-1,I-2$ we use the triangle inequality
\begin{equation}\label{triangle}
\lVert y -  y_h \rVert_{L^{\infty}(\Omega_J)} \leq \lVert y - I_h  y \rVert_{L^{\infty}(\Omega_J)} + \lVert I_h y -  y_h \rVert_{L^{\infty}(\Omega_J)}.
\end{equation}
The first term on the right-hand side can be treated with \eqref{LinftyIntI},
taking into account the relation $2-\gamma \ge 2\mu$. This implies
\begin{equation*}
\lVert y - I_h  y \rVert_{L^{\infty}(\Omega_J)}  \leq c  h^{(2-\gamma)/\mu} \lvert y \rvert_{W^{2,\infty}_{\gamma}(\Omega'_J)}\leq c  h^2\lvert y \rvert_{ W^{2,\infty}_{\gamma}(\Omega'_J)}.
\end{equation*}
We estimate the second term on the right-hand side of \eqref{triangle} by applying the inverse inequality from Lemma \ref{lemma_v_est}, and get
\[
\lVert I_h y -  y_h \rVert_{L^{\infty}(\Omega_J)} \leq c  d_J^{-1} \lVert I_h y -  y_h \rVert_{L^2(\Omega'_{J})}  \leq
 c d_J^{-1} \Big( \lVert  y - I_h y \rVert_{L^2(\Omega'_J)} + \lVert  y -  y_h \rVert_{L^2(\Omega'_J)} \Big).
\]
Finally, using \eqref{L2IntI} with $p=\infty$ we obtain
\[
d_J^{-1}\lVert  y - I_h y \rVert_{L^2(\Omega'_J)} \leq c  d_J^{-1} h^{(3-\gamma)/\mu} \lvert y \rvert_{W^{2,\infty}_{\gamma}(\Omega''_J)}
\leq c  h^{2} \lvert y \rvert_{W^{2,\infty}_{\gamma}(\Omega''_J)},
\]
where we used $d_J^{-1} h^{1/\mu} \leq d_I^{-1} h^{1/\mu} = c_I^{-1}\le c$ and the grading condition.
\end{proof}

The next lemma provides an estimate for the second terms on the right-hand sides of the estimates from Lemma \ref{lemma_inf}, the 
so-called pollution terms. 
To cover all cases $J=4,\ldots,I$, we introduce the weight function $\sigma(x) := r(x)+d_I$ 
and easily confirm that these pollution terms are bounded by 
$\lVert \sigma^{-1} (y-y_h)
\rVert_{L^2(\Omega_{R/8}})$. 
To estimate this term we can basically use the Aubin-Nitsche method involving a kick back argument.
Similar results can be found in \cite[Lemma 3.10]{APR}, where $\lVert \sigma^{-\tau} (y-y_h) \rVert_{L^2(\Omega_{R/8})}$ with $\tau = 1/2$ is considered,
or in \cite[Lemma 3.61]{Pfefferer}, where the previous estimate is generalized to exponents satisfying $1-\lambda < \tau < 1$. 
Nevertheless,  some modifications are necessary for $\tau = 1$ 

\begin{lemma}\label{lemma_l2}
Assume that $0\leq \gamma \leq 2-2\mu -2\varepsilon$
with $\varepsilon>0$ sufficiently small.
Then the estimate
\[
\lVert \sigma^{-1} (y-y_h) \rVert_{L^2(\Omega_{R/8})} \leq c \Big( h^2 \lvert \ln h\rvert \lVert y \rVert_{W^{2,\infty}_{\gamma}(\Omega_R)} + \lvert \ln h\rvert \lVert y-y_h \rVert_{L^2(\Omega_R)} \Big)
\]
is satisfied.
\end{lemma}
\begin{proof}
We define the characteristic function $\chi$, which is equal to one in $\Omega_{R/8}$ and 
equal to zero in $\Omega\setminus \text{cl} (\Omega_{R/8})$. Next, we introduce a dual boundary value problem
\begin{equation}\label{dual}
\begin{split}
-\Delta w + w &= \sigma^{-2}(y-y_h)\chi\quad  \text{in } \Omega,\\
\partial_n w &= 0\qquad\qquad\qquad\ \text{on}\ \Gamma
\end{split}
\end{equation}
with its weak formulation
\begin{equation}\label{weak}
a(\varphi,w)=(\sigma^{-2}(y-y_h)\chi,\varphi)_{L^2(\Omega)}\ \ \ \forall \varphi \in H^1(\Omega).
\end{equation}
Let $\eta\in C^\infty(\bar \Omega)$ be a cut-off function, which is equal to one in $\Omega_{R/8}$,
$\supp\eta\subset \overline\Omega_{R/4}$,
and $\partial_n \eta = 0$ on $\partial \Omega_R$, with $\lVert \eta \rVert_{W^{k,\infty}(\Omega_R)} \leq c$ for $k \in \mathbb{N}_0.$ By setting
$\varphi=\eta v$ in \eqref{weak} with some $v \in H^1(\Omega)$ one can show that $\tilde{w} = \eta w$ fulfills the equation
\begin{equation}\label{modif_a}
a_{\Omega_R}(v,\tilde{w})=(\eta \sigma^{-2}(y-y_h)\chi - \Delta \eta w - 2\nabla \eta \cdot \nabla w , v)_{L^2(\Omega_R)}\ \ \ \forall v \in H^1(\Omega),
\end{equation}
where the bilinear form $a_{\Omega_R}: H^1(\Omega_R) \times   H^1(\Omega_R) \rightarrow \mathbb{R}$ is defined by
\[
a_{\Omega_R}(\varphi,w):=\int_{\Omega_R} (\nabla \varphi \cdot \nabla w + \varphi w).
\]
By this we get
\begin{align}
\notag
\lVert \sigma^{-1}(y-y_h) &\rVert^2_{L^2(\Omega_{R/8})} = (\eta \sigma^{-2}(y-y_h)\chi,y-y_h)_{L^2(\Omega_R)}\\\notag
&= a_{\Omega_R}(y-y_h,\tilde{w}) + (\Delta \eta w, y-y_h)_{L^2(\Omega_R)} + 2 (\nabla \eta \cdot \nabla w, y-y_h)_{L^2(\Omega_R)}\\\notag
& \leq a_{\Omega_R}(y-y_h,\tilde{w}) + \Big( \lVert \Delta \eta w \rVert_{L^2(\Omega_R)} 
+ 2 \lVert \nabla \eta \cdot \nabla w \rVert_{L^2(\Omega_R)} \Big) \lVert y-y_h \rVert_{L^2(\Omega_R)}\\\label{afirst}
& \leq a_{\Omega_R}(y-y_h,\tilde{w}) + c  \lVert w \rVert_{H^1(\Omega_R)} \lVert y-y_h \rVert_{L^2(\Omega_R)}.
\end{align}
In the next step we estimate the first term on the right-hand side of the previous inequality. Since $\tilde{w}$ is equal to zero in $\Omega_R \setminus \overline\Omega_{R/4}$, we can
use the Galerkin orthogonality of $y -y_h$, i.e., $a_{\Omega_R}(y -y_h,I_h \tilde{w}) = a(y -y_h, I_h \tilde{w}) = 0.$ By this and an application of the Cauchy-Schwarz inequality we get 
\begin{equation}\label{cs}
 a_{\Omega_R}(y-y_h,\tilde{w}) = a_{\Omega_R}(y-y_h,\tilde{w}-I_h \tilde{w}) \leq c \sum_{J=2}^I \lVert y-y_h \rVert_{H^1(\Omega_J)} \lVert \tilde{w}-I_h \tilde{w} \rVert_{H^1(\Omega_J)}.
\end{equation}
Due to $\supp\eta\subset \overline\Omega_{R/4}$ there holds $\tilde{w}-I_h \tilde{w} \equiv 0$ in $\Omega_0$ and $\Omega_1$ provided that $h$ is sufficiently small. 
Now, using the results from the previous lemmas and distinguishing between $2 \leq J \leq I-3$ and $J = I-2,I-1,I$,
we can estimate the terms on the right-hand side of \eqref{cs}.

Let us discuss the case $2 \leq J \leq I-3$ first.
For the interpolation error of the dual solution we get from \eqref{L2IntO} with $\beta = 1 + \varepsilon$ or $\beta=1-\varepsilon$ the estimates 
\begin{align}
\lVert \tilde{w} - I_h \tilde{w} \rVert_{H^1(\Omega_J)} \leq c h d_J^{-\varepsilon-\mu}\lvert \tilde{w}\rvert_{W_{1+\varepsilon}^{2,2}(\Omega_J')},\label{H1_for_w_+eps}\\
\lVert \tilde{w} - I_h \tilde{w} \rVert_{H^1(\Omega_J)} \leq c h d_J^{\varepsilon-\mu}\lvert \tilde{w}\rvert_{W_{1-\varepsilon}^{2,2}(\Omega_J')}.\label{H1_for_w_-eps}
\end{align}
Both estimates are needed in the sequel. For the primal error we get with Lemma~\ref{lemma_H1}
\begin{equation}\label{H1_for_y}
\lVert y - y_h \rVert_{H^1(\Omega_J)} \leq c \Big( h d_J^{2 \varepsilon+ \mu} \lvert y \rvert_{W^{2,\infty}_{\gamma}(\Omega''_J)} + d_J^{-1} \lVert y - y_h \rVert _{L^2(\Omega_J')} \Big).
\end{equation}
To get an estimate for \eqref{cs} in case of $2 \leq J \leq I-3$ we multiply the first term on the right-hand side of \eqref{H1_for_y} with the right-hand side of \eqref{H1_for_w_+eps},  
and the second term with \eqref{H1_for_w_-eps}.
This leads to
\begin{align}\notag
\lVert y - y_h &\rVert_{H^1(\Omega_J)} \lVert \tilde{w} - I_h \tilde{w} \rVert_{H^1(\Omega_J)} \\\label{mult_eror}
&\leq c h^2 d_J^{\varepsilon} \lvert y \rvert_{W^{2,\infty}_{\gamma}(\Omega''_J)} \lvert \tilde{w}\rvert_{W_{1+\varepsilon}^{2,2}(\Omega_J')}+ c h d_J^{-1-\mu + \varepsilon} \lVert y - y_h \rVert _{L^2(\Omega_J')} \lvert \tilde{w}\rvert_{W_{1-\varepsilon}^{2,2}(\Omega_J')}.
\end{align}
Now, we recall the local a priori estimates from \cite[Lemma 3.9, (3.25)--(3.27)]{Pfefferer}, which yield in our case
\begin{equation}\label{a_priori_from_the_proof}
\lvert \tilde{w}\rvert_{W_{1+\varepsilon}^{2,2}(\Omega_J')} \leq \lVert F \rVert_{W_{1+\varepsilon}^{0,2}(\Omega_J'')} + \lVert \tilde{w} \rVert_{V_{\varepsilon}^{1,2}(\Omega_J'')}
\end{equation}
 with  the right-hand side of \eqref{modif_a}
\[
F:= \eta \sigma^{-2}(y-y_h)\chi-\Delta \eta w - 2 \nabla \eta \cdot \nabla w.
\]
Here, we use the weighted Sobolev space $V^{1,2}_{\varepsilon}(\Omega)$ 
containing homogeneous weights, i.\,e.,
\[
 \lVert v \rVert_{V_{\varepsilon}^{1,2}(\Omega_{R})}^2:= \|r^{\varepsilon-1}v\|_{L^2(\Omega_R)}^2 + \|r^\varepsilon \nabla v\|_{L^2(\Omega_R)}^2. 
\]
Inserting  estimate \eqref{a_priori_from_the_proof} into \eqref{mult_eror} yields
\begin{align}\notag
\lVert y - y_h \rVert_{H^1(\Omega_J)} \lVert \tilde{w} - I_h \tilde{w} \rVert_{H^1(\Omega_J)} &\leq c h^2 d_J^{\varepsilon} \lvert y \rvert_{W^{2,\infty}_{\gamma}(\Omega''_J)}
\Big( \lVert F \rVert_{W_{1+\varepsilon}^{0,2}(\Omega_J'')}+\lVert \tilde{w} \rVert_{V_{\varepsilon}^{1,2}(\Omega_J'')} \Big)\\\label{mult1}
&\quad +c h d_J^{-\mu + \varepsilon} \lVert \sigma^{-1} (y - y_h) \rVert _{L^2(\Omega_J')} \lvert \tilde{w}\rvert_{W_{1-\varepsilon}^{2,2}(\Omega_J')}
\end{align}
for $J=2,\hdots,I-3$, where we also used the fact that $d_J^{-1} \leq c \sigma^{-1}(x)$ for $x\in\Omega_J'$.

For the sets $\Omega_J$ with $J=I-2,I-1,I$ we apply Lemma \ref{lemma_H1} to get
\[
\lVert y - y_h \rVert _{H^1(\Omega_J)}  \leq c \Big( h^2  \lvert y \rvert_{W^{2,\infty}_{\gamma}(\Omega''_J)} + d_J^{-1} \lVert y - y_h \rVert _{L^2(\Omega_J')} \Big),
\]
and Lemma \ref{lemma_int} to get
\[
\lVert \tilde{w} - I_h \tilde{w} \rVert_{H^1(\Omega_J)} \leq c c_I^{\max\{0,-\mu+\varepsilon\}} h^{\varepsilon/\mu} \lvert \tilde{w}\rvert_{W_{1-\varepsilon}^{2,2}(\Omega_J')}.
\]
Moreover, the Leibniz rule using $\|\eta\|_{W^{k,\infty}(\Omega_R)}\le c$, $k=0,1,2$
and the global a priori estimate from
Lemma \ref{lemma_beta} with $\beta = 1-\varepsilon$ yield the estimate
\begin{align}\label{eq:W22_reg_w}
  \lvert \tilde{w}\rvert_{W_{1-\varepsilon}^{2,2}(\Omega_R)} &\leq c \lVert w \rVert_{W_{1-\varepsilon}^{2,2}(\Omega_R)} \leq  c\lVert  \sigma^{-2} (y - y_h) \rVert _{W^{0,2}_{1-\varepsilon}(\Omega_{R/8})}\nonumber\\
  &\leq c \lVert  \sigma^{-1-\varepsilon} (y - y_h) \rVert _{L^2(\Omega_{R/8})}.
\end{align}
Combining the last three estimates leads to
\begin{align}\notag
\lVert y& - y_h \rVert_{H^1(\Omega_J)} \lVert \tilde{w} - I_h \tilde{w} \rVert_{H^1(\Omega_J)} \\\notag
&\leq c \Big(  h^{2+ \varepsilon/\mu}  \lvert y \rvert_{W^{2,\infty}_{\gamma}(\Omega''_J)} 
+  c_I^{\max\{0,-\mu+\varepsilon\}} h^{\varepsilon/\mu} \lVert  \sigma^{-1} (y - y_h) \rVert _{L^2(\Omega_J')} \Big) \\\notag
&\quad\times \lVert  \sigma^{-1-\varepsilon} (y - y_h) \rVert _{L^2(\Omega_{R/8})}\\\label{multI}
&\leq c \Big(  h^2 \lvert y \rvert_{W^{2,\infty}_{\gamma}(\Omega''_J)} +   c_I^{\max\{-\varepsilon,-\mu\}} \lVert  \sigma^{-1} (y - y_h) \rVert _{L^2(\Omega_J')} \Big)  \lVert  \sigma^{-1} (y - y_h) \rVert _{L^2(\Omega_{R/8})},
\end{align}
where we exploited the property $\sigma^{-\varepsilon} \leq d_I^{-\varepsilon} = c_I^{-\varepsilon} h^{-\varepsilon/\mu}.$ 
Inserting inequalities \eqref{mult1} and \eqref{multI} into \eqref{cs} yields 
\begin{align}\notag
a&_{\Omega_R}(y-y_h,\tilde{w})\\\notag
&\leq c  \sum_{J=2}^{I-3} h^2 d_J^{ \varepsilon} \lvert y \rvert_{W^{2,\infty}_{\gamma}(\Omega''_J)}
\Big( \lVert F \rVert_{W_{1+\varepsilon}^{0,2}(\Omega_J'')}+\lVert \tilde{w} \rVert_{V_{\varepsilon}^{1,2}(\Omega_J'')} \Big)\\\notag
&+ c   \sum_{J=2}^{I-3} h d_I^{-\mu + \varepsilon} \lVert \sigma^{-1}( y - y_h) \rVert _{L^2(\Omega_J')} \lvert \tilde{w}\rvert_{W_{1-\varepsilon}^{2,2}(\Omega_J')}\\\label{sums}
& + c \sum_{J=I-2}^{I} \Big(  h^2 \lvert y \rvert_{W^{2,\infty}_{\gamma}(\Omega''_J)} +   c_I^{-\varepsilon} \lVert  \sigma^{-1} (y - y_h) \rVert _{L^2(\Omega_J')} \Big)  \lVert  \sigma^{-1} (y - y_h) \rVert _{L^2(\Omega_{R/8})},
\end{align}
where  we used $d_J^{-\mu+\varepsilon} \leq d_I^{-\mu+\varepsilon}$ and $\mu > \varepsilon$.
For the first two sums in \eqref{sums} we start with applying the discrete Cauchy-Schwarz inequality.
Moreover, for the first one we use a basic property of geometric series,
\begin{equation}\label{geom_ser}
\sum_{J=2}^{I-3} d_J^{2\varepsilon} \leq \sum_{J=0}^{I-1} \Big( 2^{ -2\varepsilon} \Big)^J
= \frac{1- 2^{-2 \varepsilon I}}{1- 2^{-2 \varepsilon}}
\le c(1-d_I^{2\varepsilon}) \le c
\end{equation}
with $c=(1-2^{-2\varepsilon})^{-1}$,
which implies $ \Big( \sum_{J=2}^{I-3} d_J^{2\varepsilon} \Big)^{1/2} \leq c.$
Note that the generic constant in \eqref{geom_ser} depends on $\varepsilon$, and tends to infinity for $\varepsilon \rightarrow 0$.
To treat the second sum in \eqref{sums}
we insert estimate \eqref{eq:W22_reg_w} as well as the properties
$\sigma^{-\varepsilon} \le d_I^{-\varepsilon}$ and $hd_I^{-\mu} = c_I^{-\mu}$.
This leads to
\begin{align}\notag
a&_{\Omega_R}(y-y_h,\tilde{w})\\\notag
&\leq c  h^2 \lvert y \rvert_{W^{2,\infty}_{\gamma}(\Omega_R)} \Big( \lVert F \rVert_{W_{1+\varepsilon}^{0,2}(\Omega_R)}+\lVert \tilde{w} \rVert_{V_{\varepsilon}^{1,2}(\Omega_R)} \Big)\\\notag
& +c c_I^{-\mu} \lVert \sigma^{-1} (y - y_h) \rVert_{L^2(\Omega_R)}  \lVert  \sigma^{-1} (y - y_h) \rVert _{L^2(\Omega_{R/8})}\\\label{a}
&+ c \Big( h^2 \lvert y \rvert_{W^{2,\infty}_{\gamma}(\Omega_R)} + c_I^{-\varepsilon} \lVert \sigma^{-1} (y - y_h) \rVert_{L^2(\Omega_R)} \Big)   \lVert  \sigma^{-1} (y - y_h) \rVert _{L^2(\Omega_{R/8})},
\end{align}
Due to the properties of the cut-off function $\eta$ and $\lVert r^{\varepsilon} \rVert_{L^{\infty}(\Omega)} \leq c$, $\lVert r^{1+\varepsilon} \rVert_{L^{\infty}(\Omega)} \leq c$,
one can show that
\[
\lVert F \rVert_{W_{1+\varepsilon}^{0,2}(\Omega_R)} \leq c \Big( \lVert  \sigma^{-1} (y - y_h) \rVert _{L^2(\Omega_{R/8})} + \lVert w \rVert_{H^1(\Omega_R)}  \Big).
\]
To estimate the $V^{1,2}_{\varepsilon}(\Omega_{R})$-norm of $\tilde w$ we use the trivial embedding
\[H^1(\Omega_{R}) \simeq W^{1,2}_0(\Omega_{R})\hookrightarrow W^{1,2}_{\varepsilon}(\Omega_{R}),\]
and exploit that the norms in $W^{1,2}_{\varepsilon}(\Omega_{R})$ and $V^{1,2}_{\varepsilon}(\Omega_{R})$ are equivalent for $\varepsilon>0$ \cite[Theorem 7.1.1]{Kozlov}.
Taking also into account the Leibniz rule with $\lVert \eta \rVert_{W^{k,\infty}(\Omega_R)} \leq c $, we obtain
\begin{equation}\label{w_formula}
\lVert \tilde{w} \rVert_{V^{1,2}_{\varepsilon}(\Omega_R)} \leq c\lVert \tilde{w} \rVert_{H^1(\Omega_R)} \leq c \lVert w \rVert_{H^1(\Omega_R)} \leq c \lvert \ln h\rvert \lVert  \sigma^{-1} (y - y_h) \rVert_{L^2(\Omega_{R/8})}.
\end{equation}
The last step is confirmed at the end of this proof. Using the previous results, inequality \eqref{a} can be rewritten in the following way
\begin{align}\label{last}
\notag
a&_{\Omega_R}(y-y_h,\tilde{w})\\
& \leq c \Big( h^2 \lvert \ln h\rvert \lvert y \rvert_{W^{2,\infty}_{\gamma}(\Omega_R)} + c_I^{-\varepsilon}\lVert \sigma^{-1} (y - y_h) \rVert_{L^2(\Omega_R)} \Big) \lVert  \sigma^{-1} (y - y_h) \rVert _{L^2(\Omega_{R/8})}.
\end{align}
By inserting \eqref{last} and the last step of \eqref{w_formula} into \eqref{afirst}, and dividing by $\lVert \sigma^{-1} (y - y_h) \rVert_{L^2(\Omega_{R/8})}$,
we obtain
\begin{align*}
 &\lVert \sigma^{-1} (y - y_h) \rVert_{L^2(\Omega_{R/8})} \\
 &\quad \leq c\Big( h^2 \lvert \ln h\rvert \lvert y \rvert_{W^{2,\infty}_{\gamma}(\Omega_R)} 
 + c_I^{-\varepsilon}\lVert \sigma^{-1} (y - y_h) \rVert_{L^2(\Omega_{R/8})}
  +\lvert \ln h\rvert \lVert y-y_h \rVert_{L^2(\Omega_R)} \Big).
\end{align*}
Here, we also used that $\sigma^{-1} = (r+d_I)^{-1} \leq r^{-1} \leq (R/8)^{-1} \leq c$, if $r \geq R/8.$ Finally, we get
\[
 \Big(1-c c_I^{-\varepsilon} \Big) \lVert \sigma^{-1} (y - y_h) \rVert_{L^2(\Omega_{R/8})}  
 \leq c \Big(  h^2 \lvert \ln h\rvert \lvert y \rvert_{W^{2,\infty}_{\gamma}(\Omega_R)} + \lvert \ln h\rvert \lVert y-y_h \rVert_{L^2(\Omega_R)}  \Big).
\]
By choosing the constant $c_I$ large enough, such that $c c_I^{-\varepsilon}  < 1$ holds,
the desired result follows.

It remains to prove the last step in \eqref{w_formula}. A similar proof was already given in \cite[Lemma 4.13]{Sirch}. 
There holds
\begin{align}\label{h1_estimate_w}
 \lVert w \rVert^2_{H^1(\Omega_R)} & \leq a(w,w) = (\sigma^{-2}(y-y_h)\chi,w) = (\sigma^{-1}(y-y_h),\sigma^{-1}w)_{L^2(\Omega_{R/8})}\nonumber\\
 &\leq \lVert \sigma^{-1}(y-y_h) \rVert_{L^2(\Omega_{R/8})} \lVert \sigma^{-1}w \rVert_{L^2(\Omega_{R})}\nonumber\\
 &\leq  c \lvert \ln h\rvert\lVert \sigma^{-1}(y-y_h) \rVert_{L^2(\Omega_{R/8})}\lVert w \rVert_{H^1(\Omega_{R})},
\end{align}
where in the last step we used estimate (4.36) from \cite[Lemma 4.13]{Sirch}, which is also valid for the Neumann boundary value problem. 
\end{proof}

From Lemma \ref{lemma_inf} and Lemma \ref{lemma_l2} we conclude the local estimate 
\begin{align}\label{lastcorner}
  \lVert y - y_h \rVert_{L^{\infty}(\Omega_{R/16})}
  &= \max_{J=4,\ldots,I} \lVert y-y_h \rVert_{L^\infty(\Omega_J)} \nonumber\\
  &\leq c  h^2 \lvert \ln h\rvert  \lVert y \rVert_{W_{\gamma}^{2,\infty}(\Omega_R)}+ \lvert \ln h\rvert  \lVert y - y_h \rVert_{L^2(\Omega_R)}.
\end{align}

\textbf{Case 2:} $x_0\in\Omega_{R_j/16}$ with $(\lambda_j,\gamma_j,\mu_j)$ satisfying 
\eqref{eq:th_assumptions} (ii).
The assumptions in this case imply $\omega_j\in (0,\pi/2)$.
We assume that the corner $x^{(j)}$ is located at the origin,
and drop the subscript $j$ as in the previous case. The basic idea is to apply Theorem \ref{thscott} in a local fashion, 
which can be realized with the technique from \cite[Theorem 1]{DLSW}.
First, we introduce a triangular domain $\hat{\Omega}_R$ 
(see Figure \ref{fig_polygon}) with vertices located at the points 
$(R, \varphi)$ with $ \varphi \in \{ 0, \omega\}$ and the origin.  
This construction guarantees that 
\[\dist(\partial \Omega_{R/2} \setminus  \Gamma,  \partial\hat{\Omega}_{R}\setminus\Gamma) > (\sqrt{2}-1)R/2 > 0,\]
which allows us (for sufficiently small $h$) to extend the mesh $\mathcal{T}_h \rvert_{\Omega_{R/2}} := \{ T\in \mathcal{T}_h : T \cap \Omega_{R/2} \neq \emptyset \}$  quasi-uniformly to an exact triangulation $\hat{\mathcal{T}}_h$ of $\hat{\Omega}_{R}$. 
We also introduce a smooth cut-off function $\eta_1$ such that $\eta_1=1$ in $\Omega_{R/2}$ and $\dist(\supp\eta_1, \partial\hat{\Omega}_{R}\setminus\Gamma) \ge c > 0$.
For our further considerations we define the Ritz projection of $\tilde{y}=\eta_1 y$ as follows. Let
\begin{figure}
\centering
\includegraphics[width=.4\textwidth]{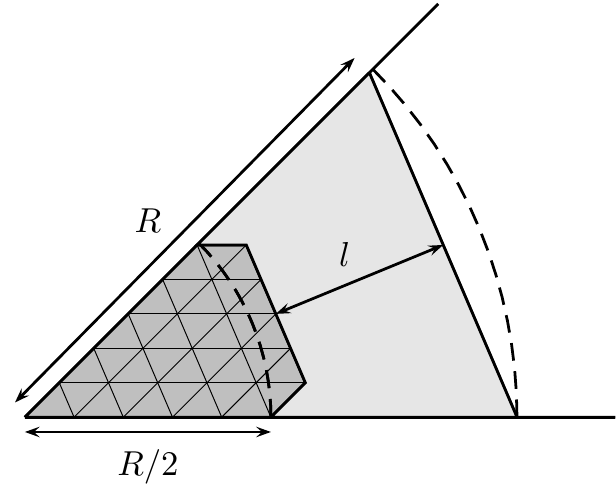}
 \caption{$\mathcal T_h|_{\Omega_{R/2}}$ - dark gray domain, $\hat{\Omega}_R$ - dark gray and light gray domains.}\label{fig_polygon}
\end{figure}
\[
V_h(\hat{\mathcal{T}}_h):=\{ v_h \in C(\cl\hat{\Omega}_R) : v_h\rvert_{T} \in \mathcal{P}_1 \ \text{for all} \ T \in \hat{\mathcal{T}}_h \}
\]
denote the space of ansatz functions with respect to the new triangulation $\hat{\mathcal{T}}_h$. The function $\tilde{y}_h \in V_h(\hat{\mathcal{T}}_h)$ is the unique solution of
\begin{equation}
a(\tilde{y}-\tilde{y}_h,v_h)=0, \quad \text{for all } v_h \in V_h(\hat{\mathcal{T}}_h). 
\end{equation}
As $y=\tilde y$ on $\Omega_{R/8}$, we get from the triangle inequality
\begin{equation}\label{y_h_tri}
\lVert y - y_h \rVert_{L^{\infty}(\Omega_{R/16})} \leq \lVert \tilde{y} - \tilde{y}_h \rVert_{L^{\infty}(\Omega_{R/8})} + \lVert \tilde{y}_h - y_h \rVert_{L^{\infty}(\Omega_{R/8})}.
\end{equation}
Due to $\tilde{y} \in W^{2,\infty}(\hat{\Omega}_{R})$, we apply Theorem \ref{thscott} and get 
\begin{equation}\label{tilde_y}
\lVert \tilde{y} - \tilde{y}_h \rVert_{L^{\infty}(\Omega_{R/8})} 
\leq c h^2 \lvert \ln h \rvert \lVert \tilde{y} \rVert_{W^{2,\infty}(\Omega_R)}
\leq c h^2 \lvert \ln h \rvert \lVert y\rVert_{W_{\vec{\gamma}}^{2,\infty}(\Omega)},
\end{equation}
where we used the Leibniz rule in the last step. Note that it is possible to construct $\eta_1$ such that $\lVert \eta_1 \rVert_{W^{k,\infty}(\Omega)} \leq c$ for $k=0,1,2$.
Next, we confirm that the function $\tilde y_h - y_h$ is discrete harmonic on $\Omega_{R/2}$, this is, for every $v_h\in V_h$ with $\supp v_h\subset\overline\Omega_{R/2}$ 
there holds
\begin{equation*}
 a(\tilde y_h - y_h, v_h) = a(\tilde y - y,v_h) = 0.
\end{equation*}
This is a consequence of $\eta_1\equiv1$ (and hence $y=\tilde y$) on $\Omega_{R/2}$, 
as well as $v_h \equiv 0$ in $\Omega\setminus\Omega_{R/2}$.
An application of the discrete Sobolev inequality \cite[Lemma 4.9.2]{Brenner_Scott} and the discrete Caccioppoli type estimate from \cite[Lemma 3.3]{Demlow}
then yield
\[
\lVert \tilde{y}_h - y_h  \rVert_{L^{\infty}(\Omega_{R/8})} \leq c \lvert \ln h \rvert^{1/2} \lVert \tilde{y}_h - y_h  \rVert_{H^1(\Omega_{R/4})}  \leq c d^{-1} \lvert \ln h \rvert^{1/2} \lVert \tilde{y}_h - y_h  \rVert_{L^2(\Omega_{R/2})},  
\]
where $d = \dist(\partial \Omega_{R/2} \setminus \Gamma, \partial \Omega_{R/4} \setminus \Gamma)$ and, by construction, $d = 1/4$ (remember $R=1$). Next, we use the triangle inequality and the fact that $y=\tilde{y}$ on $\Omega_{R/2}$. This implies
\begin{align}\notag
\lVert \tilde{y}_h - y_h  \rVert_{L^{\infty}(\Omega_{R/8})}
&\leq c \lvert \ln h \rvert^{1/2} \Big( \lVert \tilde{y} - \tilde{y}_h  \rVert_{L^2(\hat{\Omega}_R)} + \lVert y - y_h  \rVert_{L^2(\Omega)}  \Big)\\\label{tilde_y_h}
& \leq c h^2 \lvert \ln h \rvert^{1/2} \lVert \tilde y  \rVert_{H^2(\hat\Omega_R)} +   c \lvert \ln h \rvert^{1/2} \lVert y - y_h  \rVert_{L^2(\Omega)},
\end{align}
where we used a standard $L^2$-error estimate in the last step.
The estimates \eqref{tilde_y} and \eqref{tilde_y_h} finally yield
\begin{equation}\label{corner_less90}
\lVert y- y_h  \rVert_{L^{\infty}(\Omega_{R/8})}  \leq c h^2  \lvert \ln h \rvert \lVert y  \rVert_{W^{2,\infty}_{\vec{\gamma}}(\Omega)} + c\lvert \ln h \rvert^{1/2} \lVert y - y_h  \rVert_{L^2(\Omega)} .
\end{equation}

\textbf{Case 3:} This case arises when the point $x_0$ where $|y-y_h|$
attains its maximum is located in $\tilde\Omega_{R/16}$. We use \cite[Theorem 10.1]{Wahlbin} with $s=0$ to get
\[
  \lVert y- y_h  \rVert_{L^{\infty}(\tilde\Omega_{R/16})}
  \leq c \Big( \lvert \ln h \rvert \lVert y- I_h y  \rVert_{L^{\infty}(\tilde\Omega_{R/32})} + \lVert y- y_h  \rVert_{L^2(\tilde\Omega_{R/32})}  \Big).
\]
Since the domain $\tilde\Omega_{R/32}\subset\tilde\Omega_{R/64}$ has a constant and positive distance to the corners of $\Omega$, 
we conclude with standard interpolation error estimates
\begin{align}\notag
\lVert y- y_h  \rVert_{L^{\infty}(\tilde\Omega_{R/16})} 
&\leq c \Big( h^2 \lvert \ln h \rvert \lVert y \rVert_{W^{2,\infty}(\tilde\Omega_{R/64})} + \lVert y- y_h  \rVert_{L^2(\tilde\Omega_{R/32})}  \Big)\\\label{L_infty_inn}
&\leq c \Big( h^2 \lvert \ln h \rvert \lVert y \rVert_{W^{2,\infty}_{\vec{\gamma}}(\Omega)} + \lVert y- y_h  \rVert_{L^2(\tilde\Omega_{R/32})}  \Big).
\end{align}

\begin{proof}[Proof of Theorem \ref{th}] The estimates \eqref{lastcorner}, \eqref{corner_less90} and \eqref{L_infty_inn} result in 
\[
\lVert y- y_h  \rVert_{L^{\infty}(\Omega)}  \leq c h^2 \lvert \ln h \rvert \lVert y  \rVert_{W^{2,\infty}_{\vec{\gamma}}(\Omega)} + \lvert \ln h \rvert \lVert y -y_h  \rVert_{L^2(\Omega)}.
\]
For the remaining term on the right-hand side we apply 
Lemma \ref{lemma_pfefferer} for the choice $\vec\beta = 1-\vec\mu$ to conclude
the desired estimate.
Note that this choice implies  
the embedding 
\begin{equation}\label{eq:W2inf_W22_embedding}
  W^{2,\infty}_{\vec\gamma}(\Omega)\hookrightarrow W^{2,2}_{\vec\beta}(\Omega)
\end{equation}
due to $\gamma_j < 2-2\mu_j <  2-\mu_j= 1+\beta_j$, which follows
from the assumptions upon $\gamma_j$. Moreover, the required regularity
assumption holds due to $\beta_j > \gamma_j - 1 > 1-\lambda_j$.
\end{proof}

\begin{remark}\label{rem:quasi_uniform}
For quasi-uniform meshes one can show with similar arguments that the estimate
\[
	\|y-y_h\|_{L^\infty(\Omega)} \le c h^{\min\{2,\lambda-\varepsilon\}} \lvert \ln h \rvert 
   \lVert y \rVert_{W_{\vec{\gamma}}^{2,\infty}(\Omega)},
\]
holds, where the weights are chosen as 
\[
\gamma_j:=\max\{0,2-\lambda_j+\varepsilon\},\quad j=1,\ldots,m.\]
Here, $\lambda:=\min\{\lambda_j:=\pi/\omega_j\colon j=1,\ldots,m\}$ is the smallest singular exponent and $\varepsilon>0$ arbitrary but sufficiently small.
The sharpness of this convergence rate is confirmed by the numerical experiments in Section~\ref{sec:experiments}.
A detailed proof is given in \cite{Rogovs}.
\end{remark}

\section{Error estimates for semilinear elliptic problems}
The aim of this section is to extend the results from the previous section to certain nonlinear problems. To be more precise, we investigate the semilinear problem
\begin{equation}\label{semilinear}
\begin{aligned}
-\Delta y + y + d(y) &= f \quad \text{in } \Omega,\\
\partial_n y &=g \quad \text{on } \Gamma,
\end{aligned}
\end{equation}
where we assume that the input data $f$ and $g$ are sufficiently regular such that the solution $y$ belongs 
to $W^{2,\infty}_{\vec\gamma}(\Omega)$ with $\vec\gamma\in [0,2)^m$ as in
Remark \ref{rem:W2inf_regularity}. 
Under the following assumption on the nonlinearity $d$, this regularity is shown e.\,g.\ in \cite[Corollary 3.26]{Pfefferer}.
\begin{assumption}\label{assum:semi}
	The function $d :\mathbb{R}\rightarrow\mathbb{R}$, $y\mapsto d(y)$, is monotonically increasing and continuous.
	Furthermore, the function $d$ fulfills a local Lipschitz condition of the following form:
	For every $M > 0$ there exists a constant $L_{d,M} > 0$ such that
	\[
	\lvert d(y_1) -d(y_2) \rvert \leq L_{d,M}\lvert y_1 -y_2 \rvert
	\]
	for all $y_i \in \mathbb{R}$ with $\lvert y_i \rvert < M$, $i=1,2$.
\end{assumption}
\begin{remark}
	The discussion of more general nonlinearities and corresponding discretization error estimates is possible as well. In particular, the lower order term $y+d(y)$ may be replaced by a nonlinear function $\tilde d(x,y)$. Of course, in this case, further assumptions on $\tilde d$ are required, especially in order to ensure coercivity. For details we refer to \cite[Sections 3.1.2 and 3.2.6]{Pfefferer}.
\end{remark}
The variational solution of \eqref{semilinear} is a function $y\in H^1(\Omega) \cap C^0(\overline{\Omega})$ which satisfies 
\begin{equation}\label{semilinear_var}
 a(y,v) + (d(y),v)_{L^2(\Omega)}=(f,v)_{L^2(\Omega)}+(g,v)_{L^2(\Gamma)}\quad \forall v\in H^1(\Omega),
\end{equation}
where $a:H^1(\Omega)\times H^1(\Omega) \rightarrow \mathbb{R}$ is the bilinear form defined in \eqref{eq:bilinear}.
Under the assumptions on $d$, and the data $f$ and $g$, this variational formulation possesses a unique solution \cite[Theorem 4.8]{Fredi}.
Its finite element approximation $y_h \in V_h$, with $V_h$ as in Section 3, 
is the unique solution of the variational formulation
\begin{equation}\label{semilinear_var_h}
 a(y_h,v_h) + (d(y_h),v_h)_{L^2(\Omega)}=(f,v_h)_{L^2(\Omega)}+(g,v_h)_{L^2(\Gamma)}\quad \forall v_h\in V_h.
\end{equation}
Next, we show an error estimate for this approximate solution on graded triangulations
satisfying \eqref{mesh_cond}. The fundamental idea is taken from \cite[Section 3.2.6]{Pfefferer}. It is based on a supercloseness result
between the Ritz projection to the continuous solution and the finite element solution: Let $\tilde y_h\in V_h$ be the unique solution to
\[
	a(y-\tilde y_h,v_h)=0\quad \forall v_h\in V_h.
\]
By classical arguments, it is possible to show that $\tilde y_h$ is uniformly bounded in $L^\infty(\Omega)$ independent of $h$ for $f\in L^r(\Omega)$ and $g\in L^s(\Gamma)$ with $r,s>1$, see \cite[Corollary 3.47]{Pfefferer}. Moreover, Theorem~\ref{th} is applicable such that
\begin{equation}\label{eq:inftyRitz}
	\lVert y-\tilde y_h \rVert_{L^{\infty}(\Omega)}\leq c h^2 \lvert \ln h\rvert  
	\lVert y \rVert_{W_{\vec{\gamma}}^{2,\infty}(\Omega)},
\end{equation}
provided that $y$ belongs $W_{\vec{\gamma}}^{2,\infty}(\Omega)$, and $\vec\mu\in(0,1]^m$ and $\vec\gamma\in [0,2)^m$ satisfy the assumptions of Theorem~\ref{th}. The aforementioned supercloseness between $y_h$ and $\tilde y_h$ is summarized in the following lemma, taken from \cite[Lemma 3.70]{Pfefferer}. The proof essentially relies on the monotonicity and the local Lipschitz continuity of the nonlinearity $d$.
\begin{lemma}{\cite[Lemma 3.70]{Pfefferer}}\label{lem:superclose}
	Let Assumption~\ref{assum:semi} be fulfilled. Moreover, let $f\in L^r(\Omega)$ and $g\in L^s(\Gamma)$ with $r,s>1$. Then, there holds
	\begin{equation}\label{eq:superclose}
		\left\|\tilde y_h-y_h\right\|_{H^1(\Omega)}\leq c \left\|y-\tilde y_h\right\|_{L^2(\Omega)}.
	\end{equation}
\end{lemma}
\begin{remark}
	Note that, altough we only assume a local Lipschitz continuity of $d$, the constant $c$ in~\eqref{eq:superclose} is bounded independent of $h$ since $\tilde y_h$ and $y$ are uniformly bounded in $L^\infty(\Omega)$.
\end{remark}
By means of the supercloseness, it is easily possible to transfer the pointwise error estimates for linear problems to the case of semilinear problems.

\begin{theorem}
Let the assumptions of Lemma~\ref{lem:superclose} be fulfilled. Moreover, let $y\in W^{2,\infty}_{\vec\gamma}(\Omega)$ with $\vec\gamma\in [0,2)^m$. 
Then the discretization error can be estimated by
\[
 \lVert y-y_h \rVert_{L^{\infty}(\Omega)} \leq c h^2 \lvert  \ln h \rvert \lVert y \rVert_{W^{2,\infty}_{\vec{\gamma}}(\Omega)},
\]
provided that $\vec\mu\in(0,1]^m$ and $\vec\gamma\in [0,2)^m$ satisfy the assumptions of Theorem~\ref{th}. 
\end{theorem}
\begin{proof}
By introducing $\tilde y_h$ as an intermediate function, we obtain
\begin{align*}
\lVert y-y_h \rVert_{L^{\infty}(\Omega)} &\leq \lVert y-\tilde{y}_h \rVert_{L^{\infty}(\Omega)}  + \lVert \tilde{y}_h-y_h \rVert_{L^{\infty}(\Omega)} \\\label{semilin_tri}
& \leq \lVert y-\tilde{y}_h \rVert_{L^{\infty}(\Omega)}  + c \lvert \ln h \lvert^{1/2} \lVert \tilde{y}_h-y_h \rVert_{H^1(\Omega)}\\
& \leq \lVert y-\tilde{y}_h \rVert_{L^{\infty}(\Omega)}  + c \lvert \ln h \lvert^{1/2} \lVert y-\tilde{y}_h\rVert_{L^2(\Omega)},
\end{align*}
where in the last steps we used the discrete Sobolev inequality \cite[Lemma 4.9.2]{Brenner_Scott} and Lemma~\ref{lem:superclose}. The assertion finally follows from \eqref{eq:inftyRitz} and Lemma~\ref{lemma_pfefferer}.
\end{proof}

\section{Numerical example}\label{sec:experiments}
This section is devoted to the numerical verification of the theoretical convergence results of Section \ref{sec:fem}.
To this end, we use the following numerical example.
The computational domain $\Omega_{\omega}$ depending on the interior angle $\omega \in (0 , 2 \pi)$ is defined by
\begin{equation}\label{domain}
 \Omega_{\omega} := (-1,1)^2\ \cap\ \{ x \in \mathbb{R}^2 : (r(x),\varphi(x)) \in (0,\sqrt{2}] \times (0,\omega) \},
\end{equation}
where $r$ and $\varphi$ denote the polar coordinates located at the origin. In
the following, we consider the interior angles $\omega=3\pi/4$ (convex domain)
and $\omega=3\pi/2$ (non-convex domain).
To generate meshes satisfying the condition \eqref{mesh_cond}, we start with a
coarse initial mesh and apply several uniform refinement steps.
Afterwards, depending on the grading parameter $\mu$ we transform the mesh by moving all nodes $X^{(i)}$ 
within a circular sector with radius $R$ around the origin according to
\[
 X_{new}^{(i)} = X^{(i)} \Bigg( \frac{r(X^{(i)})}{R} \Bigg)^{1/\mu - 1}
\]
for all $i$ with $|X^{(i)}| < R$.
One can show that this transformation implies the mesh condition~\eqref{mesh_cond}. 
Meshes with $\mu=1$ and $\mu=0.3$ are depicted in Figure \ref{fig_meshes}.
Note that also other refinement strategies are possible.
For instance, one can successively mark and refine all elements violating \eqref{mesh_cond}.
The local refinement can be realized with a newest vertex bisection algorithm \cite{Baensch},
or a red-green-blue refinement.  
\begin{figure}
 \begin{center}
   \includegraphics[width=.35\textwidth]{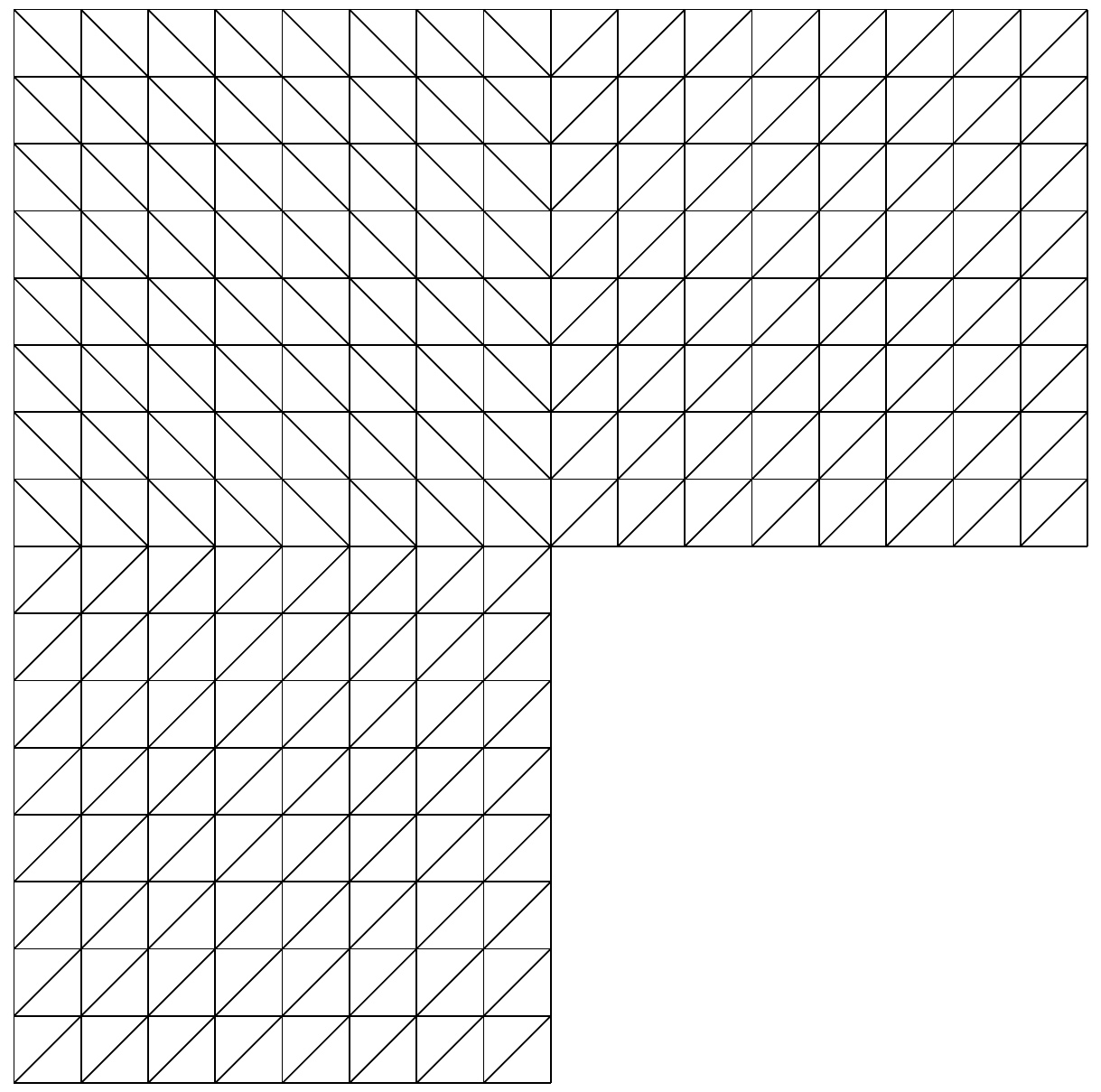}\hspace{2cm}
   \includegraphics[width=.35\textwidth]{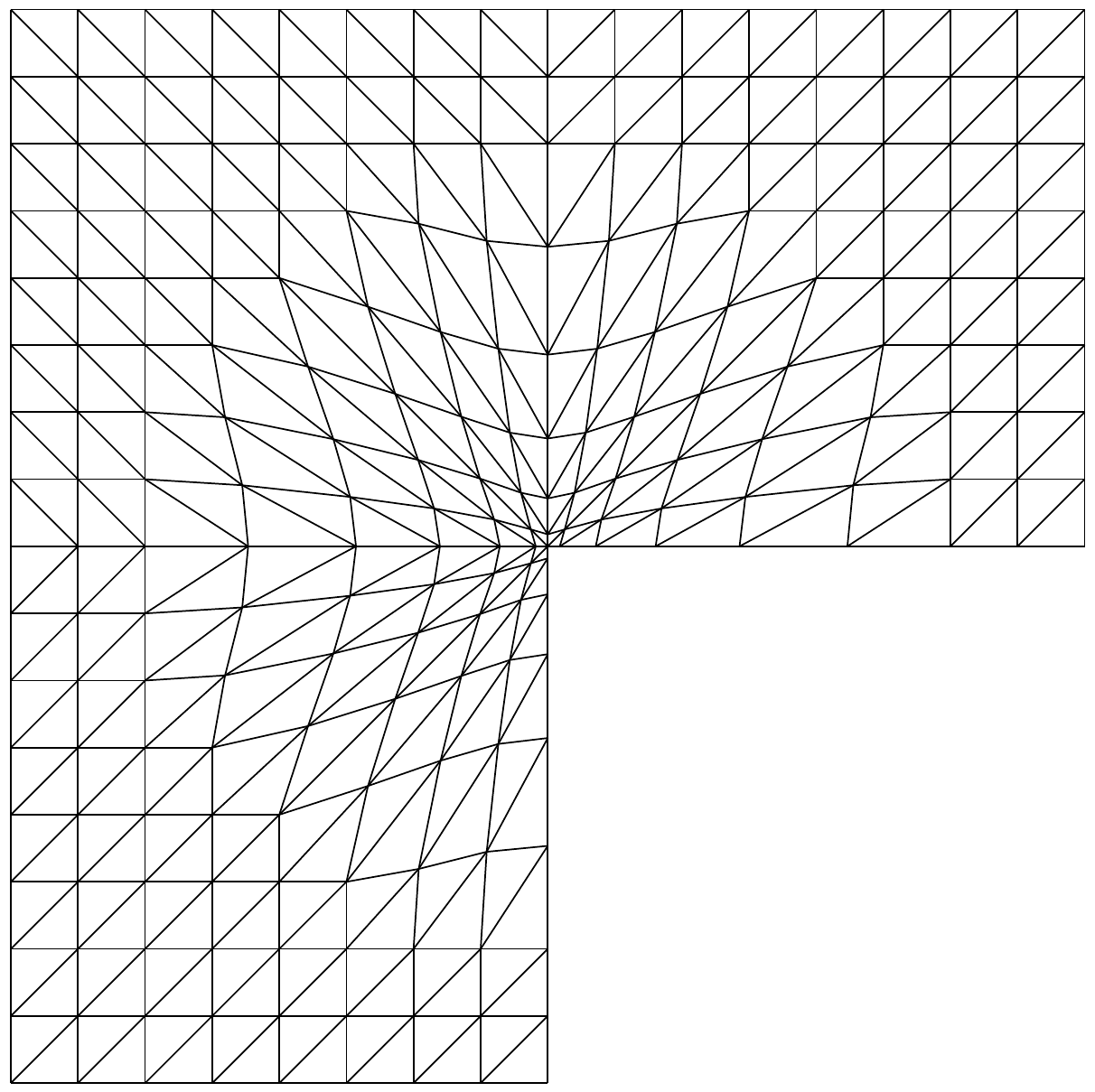}
 \end{center}
 \caption{Triangulation of the domain $\Omega_{3\pi/2}$ with a quasi-uniform ($\mu=1$) and a graded mesh ($\mu=0.3$)}\label{fig_meshes}
\end{figure}

The benchmark problem we consider is taken from \cite[Example 3.66]{Pfefferer} and reads
\begin{align*}
 -\Delta y + y &= r^{\lambda} \cos(\lambda \varphi)  &&\text{in}\ \Omega_{\omega},\\
 \partial_n y &= \partial_n \Big( r^{\lambda} \cos(\lambda \varphi) \Big) &&\text{on}\ \Gamma:=\partial\Omega_\omega,
\end{align*}
with $\lambda = \pi/\omega$. The unique solution of this problem is
$
 y= r^{\lambda} \cos(\lambda \varphi).
$
The experimental order of convergence $\text{eoc}(L^{\infty}(\Omega_{\omega}))$ is calculated by
\[
 \text{eoc}(L^{\infty}(\Omega_{\omega})):=\frac{\ln(\lVert y-y_{h_{i-1}} \rVert_{L^{\infty}(\Omega_{\omega})}/\lVert y-y_{h_i} \rVert_{L^{\infty}(\Omega_{\omega})})}{\ln(h_{i-1}/h_i)},
\]
where $h_{i-1}$ and $h_i$ are the mesh sizes of two consecutive triangulations $\mathcal{T}_{h_{i-1}}$ and $\mathcal T_{h_{i}}$. 
In Table \ref{table_3pi/4} one can find the computed errors
 $\lVert e_h \rVert_{L^{\infty}(\Omega_{3 \pi/4})} := \lVert  I_h y - y_h \rVert_{L^{\infty}(\Omega_{3 \pi/4})}$
 on sequences of meshes with $\mu=0.6 < 2/3 = \lambda/2$ and $\mu=1$.
 We measure only the discrete $L^{\infty}$-norm, since the initial error is dominated by this norm, due to
\[
\lVert y-y_h \rVert_{L^{\infty}(\Omega_{\omega})} \leq \lVert y - I_h y \rVert_{L^{\infty}(\Omega_{\omega})} + \lVert I_h y - y_h \rVert_{L^{\infty}(\Omega_{\omega})}.
\]
Note that the interpolation error is bounded by $ch^2$ if $\mu < \lambda/2$.

From our theory we expect that meshes with grading parameter $\mu < \lambda/2=2/3$ yield a convergence rate tending to 2, when the mesh size tends to zero.
For the choice $\mu=0.6$ this is confirmed. As predicted in Remark~\ref{rem:quasi_uniform} the convergence rate 
$\lambda-\varepsilon=4/3-\varepsilon$ for arbitrary $\varepsilon>0$ is confirmed for quasi-uniform meshes as well.

In Table \ref{table_3pi/2} the errors $\lVert y-y_h \rVert_{L^{\infty}(\Omega_{3 \pi/2})}$
can be found. The grading parameters are $\mu=0.3 < 1/3 = \lambda/2$, $\mu=0.6$ and $\mu=1.$
One can see that for meshes with $\mu < \lambda/2$ the convergence rate is quasi-optimal.
For meshes that are not graded appropriately,
the convergence order is not optimal too, it is about $\lambda/\mu =10/9$.
The rate $2/3-\varepsilon$ stated for quasi-uniform meshes in Remark~\ref{rem:quasi_uniform} can also be observed by the numerical experiment.

\begin{table}
\begin{center}
 \begin{tabular}{c c c c c}
   \toprule
  & \multicolumn{2}{c}{$\mu=1$} & \multicolumn{2}{c}{$\mu=0.6$}\\
  \midrule
  mesh size $h$ & $\lVert e_h \rVert_{L^{\infty}(\Omega_{\omega})}$ & eoc & $\lVert e_h \rVert_{L^{\infty}(\Omega_{\omega})}$ & eoc\\
  \midrule
  0.022097 & 1.09e-04 & 1.26 & 9.38e-05 & 1.92\\
  0.011049 & 4.50e-05 & 1.27 & 2.48e-05 & 1.94\\
  0.005524 & 1.83e-05 & 1.30 & 6.45e-06 & 1.96\\
  0.002762 & 7.39e-06 & 1.31 & 1.66e-06 & 1.97\\
  0.001381 & 2.96e-06 & 1.32 & 4.22e-07 & 1.98\\
  \bottomrule
 \end{tabular}
\caption{Discretization errors $e_h=y-y_h$ with $\omega = 3\pi/4$.}\label{table_3pi/4}
\end{center}
\end{table}

\begin{table}
\begin{center}
 \begin{tabular}{c c c c c c c}
   \toprule
  & \multicolumn{2}{c}{$\mu=1$} & \multicolumn{2}{c}{$\mu=0.6$} & \multicolumn{2}{c}{$\mu=0.3$}\\
  \midrule
  mesh size $h$ & $\lVert e_h \rVert_{L^{\infty}(\Omega_{\omega})}$ & eoc & $\lVert e_h \rVert_{L^{\infty}(\Omega_{\omega})}$ & eoc&
  $\lVert e_h \rVert_{L^{\infty}(\Omega_{\omega})}$ & eoc\\
  \midrule
  0.022097 & 6.07e-03 & 0.66 & 1.77e-03 & 1.15 & 1.44e-03 & 1.91\\
  0.011049 & 3.83e-03 & 0.67 & 8.17e-04 & 1.13 & 4.07e-04 & 1.92\\
  0.005524 & 2.41e-03 & 0.67 & 3.78e-04 & 1.12 & 1.11e-04 & 1.92\\
  0.002762 & 1.52e-03 & 0.67 & 1.75e-04 & 1.12 & 2.96e-05 & 1.95\\
  0.001381 & 9.57e-04 & 0.67 & 8.09e-05 & 1.12 & 7.70e-06 & 1.96\\
  \bottomrule
 \end{tabular}
 \caption{Discretization errors $e_h=y-y_h$ with $\omega = 3\pi/2$.}\label{table_3pi/2}
\end{center}
\end{table}

\textbf{Acknowledgment.} Supported by the DFG through the International Research Training Group IGDK 1754 ``Optimization and Numerical Analysis for Partial Differential Equations with Nonsmooth Structures''.
\bibliographystyle{plain}\bibliography{lit}
\end{document}